\documentclass[reqno]{amsart}
\usepackage{amssymb, amsmath, amsthm}
\usepackage[letterpaper, margin=1in]{geometry}
\usepackage{enumitem}
\usepackage{esint, nicefrac}
\usepackage{stmaryrd}
\usepackage[normalem]{ulem}
\usepackage{xcolor}

\usepackage{graphicx} 

\newtheorem{theorem}{Theorem}[section]
\newtheorem{lemma}[theorem]{Lemma}
\newtheorem{proposition}[theorem]{Proposition}

\theoremstyle{definition}
\newtheorem{definition}{definition}[section]
\theoremstyle{remark}
\newtheorem{remark}[definition]{Remark}
\newtheorem{assumption}[definition]{Assumption}

\numberwithin{equation}{section}

\newcommand{\norm}[1]{\left\lVert#1\right\rVert}
\newcommand{\R}{\mathbb{R}}

\newcommand{\N}{\mathbb{N}}

\DeclareMathOperator{\spt}{spt}

\DeclareMathOperator{\dist}{dist}
\newcommand\res{\mathop{\hbox{\vrule height 7pt width .3pt depth 0pt\vrule height .3pt width 5pt depth 0pt}}\nolimits}

\newcommand{\bC}{\mathbf{C}} 
\newcommand{\bG}{\mathbf{G}} 
\def\a#1{\left\llbracket{#1}\right\rrbracket}

\usepackage[colorlinks=true,linkcolor=black]{hyperref}
\usepackage{bookmark}
\title{Stationary and stable varifolds with singularities}
\author[C. De Lellis]{Camillo De Lellis}
\address{Institute for Advanced Study, 1 Einstein Drive, Princeton NJ 08540, USA}
\email{camillo.delellis@math.ias.edu}
\author[J. Hirsch]{Jonas Hirsch}
\address{Mathematisches Institut, Universit\"at Leipzig, Augustusplatz 10, D-04109 Leipzig, Germany}
\email{hirsch@math.uni-leipzig.de}
\author[L. Spolaor]{Luca Spolaor}
\address{Department of Mathematics, UC San Diego, AP\&M, La Jolla, California, 92093, USA}
\email{lspolaor@ucsd.edu}

\date{August 2025}

\begin{document}

\begin{abstract}
We construct minimal $m$-dimensional immersions in $\R^{m+1}$, equipped with a $C^{1, \alpha}$ metric, $\alpha\in [0,1)$, with a sequence of \emph{catenoidal necks} or \emph{floating disks} converging to an isolated, multiplicity $2$, singular flat point. 
\end{abstract}

\maketitle

\section{Introduction}

After the celebrated work of Allard \cite{All} on the interior singular set of stationary integral varifolds, the question of its optimal dimension and structure has attracted a lot of attention in the last 50 years, see for instance \cite{BMW,Brakke, BDF,  CCSvarifolds,DelellisAllard, DPGS,HirschSpolaor2024, HS, MenneJGA, SavinAllard,SW}. In particular, as an outcome of these works, it is clear that the main obstruction to a proof of a bound on the size of the singular set for stationary integral varifolds is the possible presence of topology on the smooth part of the varifold accumulating to flat singularities\footnote{The analogous issue for stationary varifolds which are stable on the regular part is the accumulation of classical, non immersed, singularities to a singular flat point.}. The aim of this short note is to confirm that this scenario does indeed happen when the ambient metric has low regularity, by proving the following two statements.

\begin{theorem}\label{t:uno} Fore every $\alpha \in [0,1)$, there exists a metric $g\in C^{1,\alpha}(\R^{m+1})$ and an $m$-dimensional integral current $T$ in $\R^{m+1}$, stationary with respect to $g$, such that 
\begin{enumerate}
    \item $T=\a{\Sigma}$ for a minimal immersed surface $\Sigma$ in $\mathbf{B}_1 (0) \setminus \{0\}$;
    \item the varifold tangent cone to $\Sigma$ at $0$ is $2 \a{\R^m\times \{0\}}$;
    \item the genus of $\Sigma \cap \mathbf{B}_r (0)\setminus \{0\}$ is infinite for every $r>0$.
\end{enumerate}
\end{theorem}


\begin{theorem}\label{t:due} For every $\alpha \in [0,1)$, there exist:
\begin{itemize}
    \item[(a)] a metric $g\in C^{1,\alpha}(\R^{m+1})$; 
    \item[(b)] an $m$-dimensional integral varifold $V$ in $\R^{m+1}$, stationary with respect to $g$, 
    \item[(c)] a sequence of pairwise disjoint embeddings $\gamma_k$ of $\mathbb S^{m-1}$,
    \item[(d)] and an immersed stable minimal surface $\Sigma$ in $\mathbf{B}_1 (0) \setminus (\bigcup_k \gamma_k \cup \{0\})$,
\end{itemize}    
such that $V= \Sigma$ and
\begin{enumerate}
    \item $\gamma_k$ accumulates towards the origin and $\sum_k \mathcal{H}^{m-1} (\gamma_k) < \infty$; 
    \item the tangent cone to $V$ at $0$ is $2 \a{\R^m\times \{0\}}$;
    \item for every $k$ there is a neighborhood of $\gamma_k$ in which $\Sigma$ consists of three connected components meeting at $\gamma_k$ with equal angles.
\end{enumerate}
\end{theorem}

Note that in both cases the metric is at only $C^{1,\alpha}$ for: indeed we are not able to control the Ricci curvature around the origin. This is likely only a limitation of our method of proof and the metric in our specific example is however not too pathological: it is smooth outside the origin, geodesically complete, and the geodesics spray is unique also at the origin.

Moreover, a famous open problem is whether there exists an example as in Theorem \ref{t:uno} in which $\Sigma$ is an {\em embedded} minimal surface. Our method can be adapted to prove embeddedness of $\Sigma$ if the codimension of the surface is at least $2$. 

The surface $\Sigma$ of Theorem \ref{t:due} satisfies the assumptions of the recent beautiful work \cite{Bellettini2025} of Bellettini, \footnote{Strictly speaking the work \cite{Bellettini2025} assumes $g$ to be the Euclidean metric, but it is clear that the  ``pure PDE'' approach od the work goes through with some appropriate amount of regularity for the metric $g$.} in the open set $\mathbf{B}_1 (0)\setminus (\bigcup_k \gamma_k \cup \{0\})$, except for the regularity of the metric at $0$. In \cite{Bellettini2025} the author proves that, if $\Sigma$ is a stable immersion outside a set with finite $\mathcal{H}^{n-2}$ measure\footnote{The optimal assumption in \cite{Bellettini2025} actually that the singular set has zero $2$-capacity in $\overline{\Sigma}$}, then it is locally the union of finitely many minimal graphs: Theorem \ref{t:due} suggests therefore that $\mathcal{H}^{n-1}$-finiteness would not be enough\footnote{In the multiplicity $2$ case $\mathcal H^{n-1}$-measure zero of the singular set of the immersion or ruling out triple junction singularities would be enough by \cite{MiWi}, however our example shows that this is not possible in general}. The example in Theorem \ref{t:due} is in fact a current mod $3$ which is locally area-minimizing outside of the origin, in the sense that for every point $p\neq 0$ there is a an open neighborhood $U_p$ such that $T\res U_p$ is area-minimizing mod $3$\footnote{This property follows from a simple modification of the argument given in \cite[Proof of Proposition 1.8]{DHMSS}.}. However it is certainly not area-minimizing in any neighborhood of $0$ given that the density at $0$ is $2$. 

Besides the optimality of the results in \cite{Bellettini2025}, our interest  in Theorem \ref{t:uno} comes from the famous long-standing problem of improving Allard's interior regularity theorem (cf. \cite{All}) and the recent results \cite{BMW,HS} in that direction. In particular in the recent work \cite{HS}, the second and third authors show that a stationary hypercurrent with density everywhere $<3$ is a classically embedded minimal surface outside of a set of $\mathcal{H}^m$-measure zero. In the proof, a crucial step is to show that around every point where the varifold tangent and the current tangent are a double copy of a plane, in regions without density $2$ points, the current is the locally the union of two minimal graphs (without self-intersections). One might wonder whether just requiring that the varifold tangent is a double copy of a plane, the same conclusion would hold: Theorem \ref{t:uno} provides a negative answer. The work \cite{BMW} proves that, if the topology of the multiplicity $1$ part of a stationary varifold $V$ around a multiplicity $2$ point $p$ is trivial in an appropriate sense, then the stationary varifold is a two-valued Lipschitz graph with ``small'' intersection set: again Theorem \ref{t:uno} hints at the difficulty of verifying this topological assumption.

\subsection{Strategy of the proof} The overall idea of the argument for Theorem \ref{t:uno} is that it is possible to insert (through an appropriate ``surgery'') an infinite sequence of catenoidal necks connecting two {\em Euclidan} minimal graphs which are tangent to each other at the origin and then perturb suitably the metric to make the example a stationary varifold. Even though the ``ends'' of the catenoids become flatter at a faster rate as the dimension increases, we are at the moment unable to improve the regularity of the metric. The idea of the second theorem is similar, but rather than inserting catenoidal necks we insert ``floating disks'', see Section \ref{sec.the models} for the precise description. The floating disks are in fact the union of two catenoids and a disk meeting in a common round sphere at equal angles, and their regular part is stable. Since the object has the same ends of the classical catenoid, the overall qualitative behavior of this model is indeed the same and the argument for Theorem \ref{t:due} is therefore rather similar. In order to prove the stability of the regular part of $\Sigma$ in Theorem \ref{t:due}, our construction will not only produce a metric with respect to which $\Sigma$ is stationary, but it will also produce a calibration form in a horned tubular neighborhood of the regular part of $\Sigma$.

 \subsection{Notation} Throughout the paper we will use classical notion in geometric measure theory, see e.g. \cite{Simon}. Moreover:
\begin{itemize}
\item $\mathbf{B}_r (p)$ will denote the ball of radius $r$ and center $p$ in $\R^{m+1}$;
\item $B_r (p)$ will denote the $m$-dimensional disk $\mathbf{B}_r (p) \cap \{x_{m+1}=0\}$, where it is implicitly assumed that $p\in \{x_{m+1}=0\}$;
\item $\mathbf{C}_r (p)$ will denote the cylinder with base $B_r (p)$ and vertical axis $\{(0, ..., 0, \lambda) : \lambda \in \mathbb R\}$; the point $p$ will be omitted if it is the origin;
\item $\mathbf{G}_u$ will denote the current induced by the graph of a Lipschitz function $u$ (in fact the function will typically be regular).
\end{itemize}

\subsection{Acknowledgments} ``The research of the first author has been supported by the Simons Foundation through the Simons Initiative on the Geometry of Flows (Grant Award ID BD-Targeted-00017375, UB). The third author is grateful for the support of the NSF Career Grant DMS-2044954. We are grateful to G. Orriols for pointing out a mistake in the higher regularity of the metric in the original version of this manuscript.

\section{the models}\label{sec.the models}
In this section we describe, the three models: the ``base model'' consisting of the union of two suitable minimal graphs, and the two ``local models'' for the surgeries. In all three cases we will first present the two-dimensional models and then the higher-dimensional versions.

\subsection{The base in $m=2$}\label{subsec.base m=2}
Let $N\in \N$ large be fixed and consider the ``Weierstrass'' data 
\[  
\mathbb C \ni z\; \mapsto\;  \phi(z)= \left(Nz^{N-1}, \frac{1}{2}( 1- h(z)), \frac{i}{2}(1+h(z))\right)\,.
\]
where the holomorphic function $h$ is chosen such that $\phi^2=0$, i.e. $2h(z)= N^2 z^{2N-2}$. The associated minimal surface $\Sigma_N$ is given by the image of 
\begin{align*}
X(z) &= \Re\left(\int^z \phi(z)\, dz \right)=\Re\left( \left(z^ N, \frac12\left(z-\frac{N^2}{4N-2}z^{2N-1}\right), \frac{i}{2}\left(z+\frac{N^2}{4N-2}z^{2N-1}\right) \right) \right)\\
&=\left(r^N\cos(N\theta), z+\psi(z)\right)
\end{align*}
for some smooth harmonic function $\psi(z)$ satisfying $|z||\nabla \psi|+|\psi|\lesssim |z|^{2N-1}$\,. 
This implies that $\a{\Sigma_N\cap \bC_{1}}= \bG_u$ for some (real analytic) function $u$ satisfying 
\begin{equation}\label{eq.base gradient}
    |u(z)- \Re\left(z^N\right)|+ r|\partial_z u - z^{N-1}|\lesssim |z|^{2N}\,.
\end{equation}
Consider the reflection $\boldsymbol{s}$ through the hyperplane $\{y_{m+1}=0\}$, i.e. $\boldsymbol{s}(y)=y- 2y_{m+1}e_{m+1}$. 
Our base model will then be the current
\[
\Gamma_N:=\a{\bG_u} +\boldsymbol{s}_\sharp \a{\bG_u}\,.
\]
in the case of Theorem 1 and 
\[
\a{\bG_u} -\boldsymbol{s}_\sharp \a{\bG_u}
\]
in the case of Theorem \ref{t:due}.
  
Notice that the support of $\Gamma_N$ is the union of the graphs of $u$ and $-u$, in particular it is an immersed surface. Moreover $\Gamma_N$ has a flat singular point at $0$, where its varifold tangent cone is $2\a{\pi_0}$, and has classical singularities along $2N$ half-lines, that is on the set where $u=0$. Note moreover that the map $\mathbf{s}$ reverses orientation, in particular at the point $0$ the two currents $\a{\bG_u}$ and $\boldsymbol{s}_\sharp \a{\mathbf{G}_u}$ have tangent planes with same support and opposite orientation: the ``current'' tangent to $\Gamma_N$ is therefore the trivial current $0$.

\subsection{The base in $m>2$}\label{subsec.base m>2}
In the case $m>2$ we have two possibilities.  One is to just take the two dimensional model and cross it with some $\R^{m-2}$, more precisely we  could consider the $m$-dimensional current $T=\Gamma_N \times \a{\pi}$ for the oriented plane $\pi= \{x: x_1=x_2= x_{m+1} =0\}$. We note that we still have that $T$ is the immersion of two graphs, i.e. $T=\a{\bG_u} +\boldsymbol{s}_\sharp \a{\bG_u}$, where we have extended the function $u$ constantly along $\R^{m-2}$. The difference to the two dimensional case is that the intersection of the two graphs is not transversal along $\{0\}\times \R^{m-2} \subset \R^{m+1}$.

To obtain a model $\a{\bG_u}+\boldsymbol{s}_\sharp\a{\bG_u}$ with only transversal intersections outside the origin, we would like to find a solution $u$ to the minimal surface equation that is modeled on a $N$-homogeneous harmonic polynomial $h_N$, with $\nabla h_N(y) \neq 0 $ for all $y \neq 0$. To find such a polynomial in three dimensions we can make use of the spherical harmonics and take for instance $h_N= |x|^{l(l+1)}Y^0_l(\frac{x}{|x|})$. Since $Y^0_l(\theta, \varphi)$ is $\varphi$ invariant, i.e. $\partial_\varphi Y^0_l=0$, we deduce that the nodal domains of $Y^0_l$ are $\varphi$-invariant and so the connected components of $\{Y^0_l>0\}$ and $\{Y^0_l<0\}$ alternate. Hence the Hopf-maximum principle implies that $\nabla Y^0_l\neq 0$ on the level sets $\{Y^0_l=0\}$. Hence $\a{\bG_{h_N}} + \boldsymbol{s}_\sharp \a{\bG_{h_N}}$ has only transversal intersections.

It remains to show that for any $N$-homogeneous harmonic polynomial there is a solution $u$ to the minimal surface equations such that $\lim_{r\downarrow 0} \frac{u(ry)}{r^N}=h_N(y)$. This can be achieved either following the arguments of \cite[Theorem 1.1, Corollary 1.2]{CaffarelliHardtSimon1984} or as follows. Let $A(\nabla u)\colon D^2u= \left(I - \frac{\nabla u \otimes \nabla u}{1+|\nabla u|^2}\right) \colon D^2u$ be the minimal surface operator. Since $p \mapsto A(p)$ is analytic we can appeal to the Cauchy–Kovalevskaya theorem to find a solution $u$ to $A(\nabla u)\colon D^2u=0$ in a neighborhood of $\{x_{m+1}=0\}\cap B_1$ with $u=h_N, \partial_{m+1}u=\partial_{m+1}h_N$ on $\{y_{m+1}=0\}\cap B_1$. We claim that $\lim_{r\downarrow 0} \frac{u(ry)}{r^N}=h_N(y)$. Since $u$ is analytic with $\nabla u(0)=0$ there is $k>1$ such that $\lim_{r\downarrow 0} \frac{u(ry)}{r^L} = k(y)$ for some non-trivial homogeneous harmonic function $k$. We must have $L\le N$ since on $\{y_{m+1}=0\}$ we have $\frac{u(ry)}{r^L}=\frac{h_N(ry)}{r^L}$ and the right hand side would blow up as $r\downarrow 0$. If $L <N$ then we would deduce from  $\frac{u(ry)}{r^L}=\frac{h_N(ry)}{r^L}$, $\partial_{m+1}\frac{u(ry)}{r^L}=\partial_{m+1}\frac{h_N(ry)}{r^L}$ on $\{y_{m+1}=0\}$, that $k=0$ and $\partial_{m+1} k=0$ on $\{y_{m+1}=0\}$. This implies that $k \equiv 0$, contradicting the non-triviality of $k$. We thus conclude that $L=N$, and so we obtain on $\{y_{m+1}=0\}$
\[ k(y) = \lim_{r\downarrow 0} \frac{u(ry)}{r^L}=h_N(y),\quad \partial_{m+1}k(y) = \lim_{r\downarrow 0} \partial_{m+1}\frac{u(ry)}{r^L}=\partial_{m+1}h_N(y)\,.\]
Since $h_N$ and $k$ are harmonic this implies $h_N=k$.

\subsection{The catenoid in $m=2$}\label{subsec.catenoid m=2}
The upper half of the catenoid is given by the graph of $c(z)=\cosh^{-1}(|z|)$ over the annulus $B_1^c\subset \R^2$. We will use the notation $\operatorname{Cat}$ for the current induced by it, assuming that the orientation is fixed so that the oriented tangents of the upper part of the catenoid converge to the positive orientation of the plane $\{x_{n+1}=0\}$ at infinity. We denote the scaled version by $c_\rho(z)= \rho\, c(\nicefrac{z}{\rho})$ 
 having a neck of diameter $\rho$. Given a hight $H<H_R$ there are precisely two catenoids with height $H$ at radius $R$: one stable and one unstable. They are characterized by the two solutions to 
\[\cosh^{-1}\left(\nicefrac{R}{\rho}\right)=\nicefrac{H}{\rho}\]
They can be estimated as follows: we substitute $\frac{H}{\rho}=x$ and set $f(x)=\frac{\cosh(x)}{x}$, so that the above equation reads $f(x)=\nicefrac{R}{H}$. Hence, estimating the solutions for $\rho$ reduces to estimating the ``inverse'' to $f$. 
Firstly, we note that $f'(x)<0$ for $x<x_{\min}$ and $f'(x)>0$ for $x>x_{\min} \approx 1.2$, characterized by $x_{\min}\sinh(x_{\min})=\cosh(x_{\min})$. Thus, $f$ has a unique minimum in $f(x_{\min})=\sinh(x_{\min})$ and so $f(x)=y$ has precisely two solutions for $y>\sinh(x_{\min})$. Now we may define 
\[ g(y)=\ln(2y)+\ln(\ln(2y))\le \ln(2y)+\ln(2\ln(2y))=h(y)\,.\]
Observe that for sufficient large $y$ 
\begin{align*}
    \frac{f(g(y))}{y}-1&= \frac{2y\ln(2y)+(2y\ln(2y))^{-1}}{2yg(y)}-1=-\frac{2y\ln(\ln(y)) - (2y\ln(2y))^{-1}}{2yg(y)}<0\\
    \frac{f(h(y))}{y}-1&=\frac{4y\ln(2y)+(4y\ln(2y))^{-1}}{2yh(y)}-1=\frac{2y\ln(2y)+(4y\ln(2y))^{-1}-\ln(2\ln(2y))}{2yh(y)}>0\,.
\end{align*}
Hence $f(g(y))\le y\le f(h(y))$. Since $f'(x)>0$ for $x>x_{\min}$, we then get that 
\[\ln(2y) \le g(y) \le f^{-1}(y) \le h(y)\,.\]
We can use the above to obtain an upper (and lower) bound for $\rho$: we have 
\begin{equation}\label{eq.bound for rho}
    \rho= \frac{H}{x}= \frac{H}{f^{-1}(\nicefrac{R}{H})}\le \frac{H}{g(\nicefrac RH)}\le \frac{H}{\ln(2\nicefrac RH)}\le H\,,
\end{equation}
where the last estimate holds for $R\ge \nicefrac12$. For further reference let us observe that, since $\nabla c_\rho(z)= \rho\,(|z|^2-\rho^2)^{-\nicefrac12} \frac{z}{|z|}$, we have
\begin{equation}\label{eq.derivative catenoid}
    |\nabla c_\rho(z)|\lesssim \rho \lesssim H\qquad\text{ for} |z|\ge 1\,.
\end{equation}

In summary we have that ${\eta_{0,\rho^{-1}}}_\sharp \operatorname{Cat}= \a{\bG_{c_{\rho}}\res B_{\rho}^c}+ \boldsymbol{s}^\sharp \a{\bG_{c_{\rho}}\res B_{\rho}^c}$ in the sense of currents and has height $H$ on $\partial \bC_R$.

\subsection{The ``floating disk'' in $m=2$}\label{subsec.floating disk m=2}
The ``floating disk'' consists of two pieces of the catenoid and a flat disk closing the ``neck''. The pieces of the catenoid are chosen so that they meet with an angle of $\nicefrac{2\pi}{3}$ at a common horizontal circle. 
More precisely, let $r_0$ be be the unique solution to $(\cosh^{-1})'(r_0)=\tan(\nicefrac{\pi}{3})$, i.e. $r_0=\sqrt{1+\nicefrac{1}{\tan^2(\nicefrac{\pi}{3})}}=\nicefrac{2}{\sqrt{3}}$. Finally, let $\tilde{c}(z)=\cosh^{-1}(|z|)-\cosh^{-1}(r_0)$ and associated scaled version $\tilde{c}_{\tilde{\rho}}(z) = \tilde{\rho}\, \tilde{c}(\nicefrac{z}{\tilde{\rho}})$. The ``floating disk'' is now given by the surface $\Sigma$ which induces the integral current 
\[  \a{\Sigma} = \a{B_{r_0}\times \{0\}}+ \a{\bG_{\tilde{c}}\res B_{r_0}^c} - \boldsymbol{s}_\sharp\a{\bG_{\tilde{c}}\res B_{r_0}^c}\, .\]
 In particular $\Sigma$ is the union of three minimal surfaces meeting at a common circle and note that $\a{\Sigma}$ induces a cycle mod $3$. As a consequence of \cite[Proposition 1.8]{DHMSS} the floating disk $\a{\Sigma}$ is the unique minimizer in $\bC_r$ for some $r>r_0$ as a current mod $3$. 

The corresponding rescaled version is
\[  \a{\Sigma_{\tilde{\rho}}} = (\eta_{{\tilde{\rho}}^{-1}})_\sharp \a{\Sigma}\, .\]
As for the catenoid we can solve for $\tilde{\rho}$ in terms of $R$ and $H$ again picking the ``small'' solution. Since $\tilde{c}$ deviates from $c$ only by a constant we deduce that the upper bound \eqref{eq.bound for rho} holds as well for $\tilde{\rho}$ which in particular implies as well that 
\begin{equation}\label{eq.derivative floating disc}
    |\nabla \tilde{c}_{\tilde{\rho}}(z)|\lesssim \tilde{\rho} \lesssim H\qquad\text{ for} |z|\ge 1\,.
\end{equation}

\subsection{The catenoid in $m> 2$}\label{subsec.catenoid m>2}
As in the two-dimensional case, we are looking for a surface with two ends. One approach is to search for a minimal surface that is rotationally symmetric around the $\boldsymbol{e}_{m+1}$-axis. This translates into the search for a critical point of 
\[\mathcal{A}(s)=\int_{a}^b G(s,\dot{s})\, dt \quad\text{ with } G(s,\dot{s})=|s|^{m-1}\sqrt{1+|\dot{s}|^2}\,.\]
Any stationary point of $\mathcal{A}$ must have a constant Hamiltonian $H$, i.e. for some $c>0$
\[H(s)=\frac{\partial{G}}{\partial \dot{s}}\dot{s}-G= -\frac{|s|^{m-1}}{\sqrt{1+|\dot{s}|^2}}=-c^{m-1}\,.\]
Passing to $\tilde{s}(t)=c \,s(\frac{t}{c})$ we may assume $c=1$.
Hence we reduced the second order ODE to the following first order one:
\[ \frac{\dot{s}}{\sqrt{s^{2(m-1)}-1}}=1\,.\]
Since for $0\le s-1$ small we have $s^{2(m-1)}-1 = 2(m-1)(s-1) + O((s-1)^2)$ or $\frac{1}{\sqrt{s^{2(m-1)}-1}}\le \frac{2}{\sqrt{2(m-1)(s-1)}}$ and for large $s$ we have $\frac{1}{\sqrt{s^{2(m-1)}-1}}\le\frac{1}{2s^{m-1}}$ we deduce that $\frac{1}{\sqrt{s^{2(m-1)}-1}}$ is integrable on $(1,\infty)$. Thus we can define 
\[c(r)=\int_{1}^r \frac{1}{\sqrt{s^{2(m-1)}-1}}\, ds \quad\text{ with }\quad c_\infty = \lim_{r\to \infty} c(r)<\infty\,.\]
As in the the two dimensional case, given $H$ small we are interested in finding $\rho\le R$ such that 
\begin{equation}\label{eq.finding rho}
    h(\rho):=\rho \,c\left(\frac{R}{\rho}\right)=H\,.
\end{equation}
Since 
\[h'(\rho) = \int_{1}^{\frac{R}{\rho}} \frac{1}{\sqrt{s^{2(m-1)}-1}}\, ds - \frac{R}{\rho\sqrt{(\nicefrac{R}{\rho})^{2(m-1)}-1}}\]
has precisely one zero $\rho_0$ hence $h(\rho)$ has a unique maximum $H_\infty$. Furthermore the equation \eqref{eq.finding rho} has two solutions for $H<H_\infty$. As in the two dimensional situation we take the smaller one. 
Let $R_\infty>0$ be fixed such that $\frac{c_\infty}{2} \le c(R_\infty) \le c_\infty$ and $\rho_1<R$ such that $\frac{R}{\rho_1}=R_\infty$. Now given $H$ small we observe that
\begin{align*}
    h(\rho) &\le h\left(\frac{H}{c_\infty}\right) \le H \quad&\text{ for }& \quad \rho \le \frac{H}{c_\infty}\\
    h(\rho) &\ge \frac{2H}{c_\infty} \frac{c_\infty}{2} \ge H \quad&\text{ for }& \quad \frac{2H}{c_\infty} \le \rho \le \frac{R}{R_\infty}\,.
\end{align*}
Hence we deduce that the smaller solution to \eqref{eq.finding rho} satisfies $\rho \approx H$. This implies that our rescaled catenoid $c_\rho(r)=\rho \, c(\nicefrac{r}{\rho)})$ satisfies for $H \ll R$ 
\[\left(\frac{\rho}{R}\right)^{m-1} \le \dot{c}_\rho(R)=\frac{\rho^{m-1}}{\sqrt{R^{2(m-1)}-\rho^{2(m-1)}}} \le 2  \left(\frac{\rho}{R}\right)^{m-1}\,.\]
In summary we found
\begin{equation}\label{eq.estimates for catenoid in m>2}
    \rho \approx H \;\text{ and }\; |\nabla c_\rho(R)| \lesssim H^{m-1}\;.
\end{equation}

\subsection{The ``floating disk'' in $m>2$}\label{subsec.floating disk m>2}
The construction follows the same lines as in the two dimensional case. For us the only important facts are 
\begin{equation}\label{eq.estimates for floating disk m>2}
    \tilde{\rho} \approx H \;\text{ and }\; |\nabla \tilde{c}_\rho(R)| \lesssim H^{m-1}\;.
\end{equation}

\begin{remark}\label{r:loss-in-2d}
We want to emphasize that the bound in \eqref{eq.derivative catenoid} is linear in $H$ while in the higher-dimensional case the corresponding bound in \eqref{eq.estimates for catenoid in m>2} is superlinear. However the reason for the low regularity of the metric $g$ is in fact the estimate on the derivative of the base model, which is not flat enough.
\end{remark}

\section{The calibration form and the metric in the gluing region}\label{sec.the calibration form}

 The aim of this section is to give a way to transition smoothly, in an annulus, from one minimal graph (e.g. the one given the base model) to another (e.g. one given by a local model) and hence perturb the metric so that they remain stationary and stable. The transition from one graph to the other will be achieved with a simple partition of unity. In order to change the metric so to keep stationarity and stability we will construct a suitable ``calibration''-form $\omega$ in a tubular neighborhood of the obtained graph $\bG_v$. The form will then need to satisfy $d\omega=0$ and $\omega|_{\bG_v}= d\operatorname{vol}_{v}$.
Furthermore, we want to estimate the derivatives of $\omega$. The construction of the calibration is inspired by \cite{DDH}. 

Let us introduce the following notation for an annulus $\Omega_{r,s}=B_{s}\setminus B_r \subset  \R^m$. Moreover we will often use the gluing function $\theta \colon \R\to [0,1]$, which is a smooth function identically equal to $0$ for $t<1/3$ and to $1$ for $t>1-1/3$, and its translated version $\theta_k(t):=\theta(t-k)$.

We will work under the following assumptions:

\begin{assumption}\label{ass:glue}
Let $u_i\colon \Omega_{7,12}\to \R$ be two minimal graphs that satisfy
\begin{enumerate}[label=(A.\arabic*)]
    \item\label{assumption.graph1} there is $x_0 \in \Omega_{7,12}$ with $u_1(x_0)=u_2(x_0)=0$
    \item\label{assumption.graph2} $\norm{\nabla u_1}_{L^\infty(\Omega_{7,12})}+\norm{\nabla u_2}_{L^\infty(\Omega_{7,12})} = \boldsymbol{\delta}<\boldsymbol{\delta}_0$.
\end{enumerate}
\end{assumption}

Furthermore, we will denote with $N_{u_i}= \frac{1}{\sqrt{1+|\nabla u_i|^2}} (-\nabla u_i, 1)$ the normals to the graphs of $u_i$ implicitly considering each normal as a function on $\Omega_{7,12}\times \R$ by $N_{u_i}(z,x_3)=N_{u_i}$. With their help we obtain calibration forms in the Euclidean metric for $\bG_{u_i}$ in $\Omega_{7,12}\times \R$: 
\begin{equation}\label{eq.callibration forms for the graphs}
    \omega_{u_i}= \sum_{\alpha} N^\alpha_{u_i}\, \star dx^\alpha = \det(N_{u_i}, \cdot, \cdot) \,.
\end{equation}

Next we glue $u_1$ to $u_2$ to define a smooth function $v\colon \Omega_{7,12}\to \R$ in the following way
\begin{equation}\label{eq.gluefunct}
v(x)=(1-\theta_9(|x|))\, u_1(x)+\theta_9(|x|)\, u_2(x)\,.
\end{equation}
It follows directly from Assumptions \ref{assumption.graph1} that
\begin{equation}\label{eq.graph distant}
    \norm{u_1-u_2}_{L^\infty(\Omega_{7,12})} \lesssim \norm{\nabla (u_1-u_2)}_{L^\infty(\Omega_{7,12})} \lesssim \boldsymbol{\delta}.
\end{equation}
Moreover, standard elliptic estimates imply that  
\begin{equation}\label{eq.higher derivatives}
    \norm{D^k u_i}_{L^\infty(\Omega_{8,11})} \lesssim_k \norm{\nabla u_i}_{L^\infty(\Omega_{7,12})} \lesssim \boldsymbol{\delta} \qquad \text{for any $k\ge 1$}.
\end{equation}
Notice that, by \eqref{eq.graph distant} and \eqref{eq.higher derivatives}, we also have 
\begin{equation}\label{eq:vestimate1}
    \|D^k v\|_{L^\infty(\Omega_{8,11})}\lesssim_k \boldsymbol{\delta}\qquad \text{for any $k\ge 1$}\,.
\end{equation}

\subsection{The calibration form}

The goal of this section is to prove the following 

\begin{proposition}[Local gluing]\label{prop:localglue}
    Let $u_1,u_2$ be as in Assumptions \ref{ass:glue} and let $v$ be as in \eqref{eq.gluefunct}. Then there exists a \emph{closed} form $\omega$ defined in a neighborhood of the graph of $v$ such that
    \begin{enumerate}[label=(F.\arabic*)]
    \item\label{prop:localglue-F1} $\omega|_{\bG_v}=d\operatorname{vol}_v$, that is $\omega$ is identically one on the tangent to $\bG_v$;
    \item\label{prop:localglue-F2} $\omega=\omega_{u_1}$ on $\Omega_{7,8}\times [-10,10]$ and $\omega=\omega_{u_2}$ on $\Omega_{11,12}\times [-10,10]$;
    \item\label{prop:localglue-F3} $\|\omega\|-1 = \dist(y,\bG_v)\,R$ with $\norm{D^k R}_{L^\infty(\Omega_{7,12}\times[-10,10)} \lesssim_k \,\boldsymbol{\delta}$ for any $k$.
    \end{enumerate}
\end{proposition}

This form will eventually be ``made'' into a calibration form by choosing an appropriate perturbation of the Euclidean metric. We nonetheless abuse the term somewhat by calling it a calibration already at this stage. The proof is split in four steps.

\subsubsection*{Step 1: normal coordinates}
Let $\phi=\varphi + \tau N$ be a parametrization of the normal neighborhood of the graph of a smooth function $v$ where $\varphi(x)= x + v(x) \boldsymbol{e}_{m+1}$, and denote with $g_{ij}:=\partial_i\varphi\,\partial_j\varphi$ and with $|g|$ the determinant of this matrix. Moreover let $d\operatorname{Vol}_v$ be the volume form of $\bG_v$ and $\omega_v:=\det(N_v,\cdot,\cdot)$, as above.  Since $\partial_i \phi= \partial_i\varphi - \tau h_i^k \partial_k \varphi= b^k_i \partial_k\varphi$, where $h_i^k=h_i^k(\nabla v, D^2v)$ is the second fundamental form of $\bG_v$, and $\partial_\tau\phi = N$, the pullback metric is given by
\[ \phi^\sharp \delta_{\R^{m+1}}= \left(g_{ij} - \tau \left(g_{ik}h^k_j + g_{jk}h^k_i\right) + \tau^2 h_i^kg_{kl}h^l_j\right) \;dx^i\otimes dx^j + d\tau\otimes d\tau = k_{ij}\;dx^i\otimes dx^j + d\tau\otimes d\tau\]
Furthermore we can expand the above defined differential forms in these coordinates given by 
\begin{align}
    \phi^\sharp d\operatorname{vol}_v&= \sqrt{|g|} \,dx &\text{ with }& \| D^{k+1} \sqrt{|g|}\|_{L^\infty(\Omega_{8,11})} \lesssim _k \boldsymbol{\delta} \label{eq.forms in normal coordinates1}\\
    \phi^\sharp \omega_v &= (1 + \tau h^k_k ) \sqrt{|g|} \, dx + \tau^2 \sigma_1 + \tau \sigma_2 &\text{ with }& \norm{ D^k\sigma_1}_{L^\infty(\Omega_{8,11}\times [-10,10])}\lesssim \boldsymbol{\delta}\nonumber\\
&& \text{and}&  \norm{ D^k\sigma_2}_{L^\infty(\Omega_{8,11}\times [-10,10])}\lesssim \boldsymbol{\delta}^2.\label{eq.forms in normal coordinates2}
\end{align}
This can be seen as follows. Let $\boldsymbol{p}$ be the projection onto the support of the current $\bG_v$ (namely the set-theoretic graph of $v$). Then we have 
\begin{align*}
    \phi^\sharp \omega_v( \partial_1, \dotsc ,\partial_m)&= \det(N_v(\boldsymbol{p}(\phi)), b_1^{k_1} \partial\varphi_{k_1}, \dots, b_m^{k_m} \partial\varphi_{k_m})= \det(b)\left(\sqrt{|g|} + \tau^2 \sigma_1\right) \\
    &=(1+\tau h^k_k + \tau^2 \tilde{b}) \left(\sqrt{|g|} + \tau^2\, \sigma_1\right) \quad\text{ with } |\sigma_1|\lesssim |D^3v|, |\tilde{b}|\lesssim |D^2v|^2\\
    \phi^\sharp \omega_v(\partial_\tau, \partial_{j_1}, \dotsc, \partial_{j_{m-1}})&= \tau \sigma_{j_1,\dotsc, j_{m_1}} \quad \text{ with } |\sigma_{j_1,\dotsc, j_{m_1}}|\lesssim |Dv||D^2v|\,,
\end{align*}
where we have used that $N(\boldsymbol{p}(\phi)) = N(x) + \tau \mu(\tau,x)$ with $N(x)\cdot \mu(0,x)=0$ and $|\mu(\tau,x)|\lesssim |Dv||D^2v|$.

Finally we observe that 
\begin{align}\label{eq.norm of dx in normal coordinates}
    \norm{dx}^2_{\phi^\sharp\delta_{\R^{m+1}}} = \det\left(\langle dx^i,dx^j\rangle_{\phi^\sharp \delta_{\R^{k+1}}^{-1}}\right)= \det(k^{-1}) = \det(k)^{-1} = \left(|g| (1 + \tau h^k_k + \tau^2 a) \right)^{-1}, 
\end{align}
where $|a|\lesssim |h|^2 \lesssim |D^2v|^2\lesssim \boldsymbol{\delta}^2$ and so $\norm{D^ka}_{L^\infty(\Omega_{8,9}\times [-10,10])}\lesssim_k \boldsymbol{\delta}^2$.

\subsubsection*{Step 2: calibration in the region $\Omega_{8,9}$} In the the cylinder over $\Omega_{8,9}$ we want to pass from $\omega_v=\omega_{u_1}$ to $d\operatorname{vol}_v=d\operatorname{vol}_{u_1}$. Note that in this region $\bG_{v}$ is a minimal surface with respect to the Euclidean metric. In particular this implies that in the normal neighborhood coordinates all terms of type $\tau h^k_k$ vanish. 

Consider the closed form 
\[\sigma = \phi^{\sharp}\left(\omega_v - d\operatorname{vol}_v\right) = \tau^2 \tilde{\sigma}_1 + \tau \tilde{\sigma}_2\]
with $\sigma =0$ on $\{\tau =0\}$, and $\norm{D^k\tilde{\sigma}_1}_{L^\infty(\Omega_{8,9}\times [-10,10])}\lesssim_k \boldsymbol{\delta}, \norm{D^k\tilde{\sigma}_2}_{L^\infty(\Omega_{8,9}\times [-10,10])}\lesssim \boldsymbol{\delta}^2 $ by Step 1 \eqref{eq.forms in normal coordinates1} and \eqref{eq.forms in normal coordinates2}.
Now we consider the vector field $Y= \partial_\tau$ with associated flow $\psi_t(x,\tau)=(x,e^t\tau)$. It exists for all $t\in (-\infty, 0]$, hence we can apply Lemma \ref{lem.poincare} and obtain the following 
\[ d (I^Y\sigma) = \sigma \text{ and } I^Y\sigma = \tau^2 \mu_1 + \tau\mu_2. \]
In particular $I^Y \sigma$ satisfies 
\begin{equation}\label{eq.estimates on I^Ysigma}
    I^Y\sigma = \tau^2 \mu_1 + \tau\mu_2 \text{ with } \norm{D^k\mu_i}_{L^\infty(\Omega_{8,9}\times[-10,10])}\lesssim_k \boldsymbol{\delta}^i\,.
\end{equation}
For a fixed cut-off function $\theta=\theta(|x|)$ with $\theta=1$ for $|x|<8+\frac13$ and $\theta=0$ for $|x|> 8 +\frac23$ we obtain the desired closed form by 
\begin{equation}\label{eq.the callibration form in Omega89}
    \phi^\sharp\omega = \theta \phi^\sharp \omega_v + (1-\theta) \phi^\sharp d\operatorname{vol}_v +d\theta \wedge I^Y \sigma.
\end{equation}

Since $\phi^\sharp\omega - \phi^\sharp d\operatorname{vol}_v = \theta (\phi^\sharp\omega_v- \phi^\sharp d\operatorname{vol}_v) + d\theta \wedge I^Y\sigma$ we conclude from \eqref{eq.forms in normal coordinates1}-\eqref{eq.forms in normal coordinates2} and \eqref{eq.estimates on I^Ysigma} that 
\[ \norm{\phi^\sharp(\omega -d\operatorname{vol}_v)}_{\phi^\sharp \delta_{\R^{m+1}}}\lesssim \tau^2 \boldsymbol{\delta}+ \tau \boldsymbol{\delta}^2,\]
where we have used in particular the fact that $v$ is minimal in the region and so the linear term $\tau h^k_k$ vanishes in \eqref{eq.forms in normal coordinates2}. By the previous expansion its clear that the same holds for the derivatives.  
Furthermore \eqref{eq.norm of dx in normal coordinates} and again $\tau h_k^k =0$ imply that $\left|\norm{\phi^\sharp d\operatorname{vol}_v}_{\phi^\sharp \delta_{\R^{m+1}}}- 1 \right|\lesssim  \tau^2 \boldsymbol{\delta}^2$ hence we conclude for the norm
\begin{equation}\label{eq.norm of the callibration form in Omega89}
    \norm{\phi^\sharp \omega}_{\phi^\sharp \delta_{\R^{m+1}}}- 1 = \tau c_1 + \tau^2c_2 \text{ with } \norm{D^kc_i}_{L^\infty(\Omega_{8,9}\times[-10,10])} \lesssim_k \boldsymbol{\delta}^i\,.
\end{equation}

\subsubsection*{Step 3: calibration in the region $\Omega_{9,10}$}
In this region, we just set $\phi^\sharp\omega = \phi^\sharp d\operatorname{vol}_v$. Unfortunately, its norm due to \eqref{eq.forms in normal coordinates2} and \eqref{eq.norm of dx in normal coordinates} has the expansion 
\[ \norm{\phi^\sharp \omega}_{\phi^\sharp \delta_{\R^{m+1}}}- 1 = \tau \tilde{c} \text{ with } \norm{D^k\tilde{c}}_{L^\infty(\Omega_{8,9}\times[-10,10])} \lesssim_k \boldsymbol{\delta}\,.\]
In this region we have lost the feature that $v$ is minimal so $\tau h^k_k$ does not vanish. 

\subsubsection*{Step 4: calibration in the region $\Omega_{10,11}$}
The construction is literally the same as in the region $\Omega_{8,9}$ with the only differences that the cut-off $\theta$ this time is chosen so that  
$\theta=0$ for $|x|<10+\frac13$ and $\theta=1$ for $|x|> 10 +\frac23$ and $u_1$ is replaced by $u_2$. 
Hence $\omega$ has in this region the structure of \eqref{eq.the callibration form in Omega89} and its norm has an expansion as in \eqref{eq.norm of the callibration form in Omega89}.

\subsection{The conformal change of the  Euclidean metric}  We will now change the metric suitably so that the form of Proposition \ref{prop:localglue} becomes a calibration.

\begin{lemma}\label{lem.definition of the metric}
Let $\omega$ be as in Proposition \ref{prop:localglue}. There is a conformal change $g=\lambda^2 \,\delta_{\R^{m+1}}$ of the Euclidean metric on $\R^{m+1}$ such that 
\begin{enumerate}[label=(M.\arabic*)]
    \item\label{lem.definitioin of the metric-M1} $\lambda(y) =1 $ if $\dist(y,{\rm spt}(\bG_v))>3$;
    \item\label{lem.definitioin of the metric-M2} $\lambda(y)=\operatorname{comass} \omega(y)$ if $\dist(y, {\rm spt}\, (\bG_v))< 2$;
    \item\label{lem.definitioin of the metric-M3} $\lambda(y) -1= \dist(y,  {\rm spt}\, (\bG_v))\, R$, with $\norm{D^kR}_{L^\infty(\Omega_{7,12}\times \R)}\lesssim_k \boldsymbol{\delta}$ .
\end{enumerate}
\end{lemma}

 We notice that for a general $k$-form $\alpha$ we have 
\[ \operatorname{comass}_{\lambda^2 g}\alpha = \lambda^{-k}\operatorname{comass}_{g}\alpha\,.\]
Lemma \ref{lem.definition of the metric} gives a metric for which the closed form $\omega$ is a calibration of $\mathbf{G}_v$ in some tubular neighborhood of its support, thanks to the choice $\lambda= \operatorname{comass}_g{\omega}$. Observe that this choice is not in general a good one because the $\operatorname{comass}$ norm is not differentiable. However, the comass norm is smooth for $n$-forms in $\R^{n+1}$, because the latter are all simple: for simple forms $\operatorname{comass}_g\omega$ coincides with the usual norm induced by the Riemannian metric on the exterior product of the cotangent bundle.

\begin{proof}
    For a smooth $\theta$ with $\theta(t) \equiv 1$ for $t< 2$ and $\theta(t)=0$ for $t>3$, the interpolation
    \[ \lambda(x)=1+\theta(\dist(y,  {\rm spt}\, (\bG_v)))\left(\norm{\omega(y)}-1\right)\]
    has the desired properties, by Proposition \ref{prop:localglue}.
\end{proof}

\section{The construction of the example}

\subsection{A single gluing step}\label{subsec.single gluing}

In this section we glue, at a fixed distance from the origin, one between the floating disk and the catenoid to the base model. The key point is that the conformal factor for the new metric obtained through Lemma \ref{lem.definition of the metric} is invariant under reflection through $\{y_{m+1}=0\}$. Let us note that the estimate in \ref{list.properties of lambda} states that the derivatives of the conformal factor are bounded up to the  order $m+1$.

\begin{proposition}[Single gluing step]\label{prop:single gluing} Let $u, c_\rho$ and $\tilde{c}_\rho$ be as in Section \ref{sec.the models}. Let $0<r\leq r_0^{N-\frac{N-1}{m}}$ and let $u_2=\frac{u(r_0+rx)}{r}$ and $u_1$ be either $c_\rho$ or $\tilde{c}_\rho$, with $\rho$ chosen as in Section \ref{sec.the models}, with parameter $H=u_2(10)$ so that 
    \[ H \approx \frac{r_0^N}{r} \text{ and } \rho,\tilde{\rho} \lesssim  H\,.\]
Then there are a function $v\colon B_{12}\setminus B_{\theta H}\to \R$, two closed forms $\omega_\pm$ defined on $\{\dist(y,\bG_{\pm v})<5\}\cap \bC_{12}\setminus \bC_{\theta H}$ and a conformal factor $\lambda_2\colon \bC_{12}\to \R$ such that the following properties are satisfied.
\begin{enumerate}[label=(G.\arabic*)]
    \item\label{list.properties of v} $v=u_2$ in $B_{12}\setminus B_{11}$ and $v=u_1$ in $B_8\setminus B_{\theta H}$, and moreover 
     $\norm{D^k(v-u_2)} \lesssim H^{m-1}$ in $\Omega_{7,12}$. We will consider the surface
     \begin{align*}
    \tilde{\Gamma}_{r_0,r}=\begin{cases}  \a{\bG_{v}\res B_{\theta H}^c} + \boldsymbol{s}_\sharp\a{\bG_{v}\res B_{\theta H}^c} &\text{ in case of the catenoid }\\
    \a{B_{\theta H}\times \{0\}} \cup \a{\bG_{v}\res B_{\theta H}^c} \cup \boldsymbol{s}_\sharp\a{\bG_{v}\res B_{\theta H}^c} &\text{ in case of the ``floating disk''}.
    \end{cases}
\end{align*}
    \item\label{list.properties of lambda} $\lambda_2=1$ inside of $\bC_{8}$ and outside $\bC_{11}$, $\lambda_2=\lambda_2\circ \boldsymbol{s}$ and moreover for any $\alpha\in \N_0^m, k\in \N_0$ \[|\partial_x^\alpha \partial_{y_{m+1}}^k(\lambda_2(y)-1)| \lesssim_{k+|\alpha|} H^{m-k} \]
    We will consider the conformal change of the euclidean metric given by $g=\lambda_2^2 \delta_{\R^3}$ in $\bC_{12}$.
    \item\label{list.properties of omega} $\omega=\omega_{u_1}$ on $\bC_8\setminus \bC_{\theta H}$\footnote{Note that $\omega_{u_1}$ is defined in the full cylinder $\bC_{\theta H}^c$.} satisfying $\omega|_{\bG_{v}} = d\operatorname{vol}_{v}$ in  $\bC_{12}\setminus \bC_{\theta H}^c$ and $\operatorname{comass}_g(\omega(y))=1$ if $y \in \bC_{12}\setminus \bC_{\theta H}^c$  and $\dist(y,\bG_{ v})< 2$
\end{enumerate}
\end{proposition}

\begin{proof}
Recall that we would like to connect the two sheets of $\Gamma_N$  to one of our two local models, the catenoid or the floating disc, inside the cylinder $\bC_{12 r}(r_0 e_1)$. To do so, we will rescale via $\eta_{r_0e_1,r}$ to size 1  (and eventually rescale back).


Firstly, let us set $u_2(x)= \frac{u(r_0+ rx)}{r}$, i.e. $\bG_{u_2}\res \bC_{12} = (\eta_{r_0 e_1, r})_\sharp \bG_{u}\res \bC_{12 r}(r_0 e_1)$.
Observe that \eqref{eq.base gradient} implies that with respect to the homogeneous harmonic polynomial $h$ of order $N$ we have
\begin{equation}\label{eq.base gradient-rescaled}
    \left|u_2(x)- \frac{h(r_0+rx)}{r}\right|+ r|\nabla u_2(x) - \nabla h(r_0+rx)|\lesssim r_0^{2N}\, \quad \forall x \in B_{12}\,.
\end{equation}
In particular this implies that  $u_2 \ge c H$ in $B_{12}$. Now we want to pass from $u_2$ to $c_\rho$ or $\tilde{c}_{\tilde{\rho}}$ in $\Omega_{8,11}$ to obtain a new function $v$ defined in $B_{12}\setminus B_{\theta H}$ where $\theta H\in \{\rho, \tilde{\rho}\}$ is the radius of the hole, or floating disk respectively. 
To be more precise, 
let us fix $\rho, \tilde{\rho}$ according to section \ref{subsec.catenoid m=2}, \ref{subsec.catenoid m>2}, \ref{subsec.floating disk m=2} and \ref{subsec.floating disk m>2}, with parameters $H=u_2(10)$ and $R=10$. In particular we have 
\[ H \approx \frac{r_0^N}{r} \text{ and } \rho,\tilde{\rho} \lesssim  H\,.\]
Moreover we have 
\[ \frac{c_\rho(r_1)}{c_\rho(r_2)}+ \frac{\tilde{c}_\rho(r_1)}{\tilde{c}_\rho(r_2)} \le C \quad \forall r_1,r_2 \in [7,12] \]
In particular this implies that $c_\rho(x), \tilde{c}_\rho(x) \ge c H$  for some $c>0$ and $7\le |x|\le 12$.
Finally we recall that,

\begin{align*}
    \norm{\nabla u_2}_{L^\infty}+\norm{\nabla c_{\rho}}_{L^\infty}+\norm{\nabla\tilde{c}_{\tilde{\rho}}}_{L^\infty} &\lesssim H^{m-1}\,.
\end{align*}
In particular, the estimate on $u_2$ follows from our choice of $r$ in comparison to $r_0$, which guarantees that $\frac{\norm{\nabla u}_{L^\infty}}{H^{m-1}} \approx \frac{r^{m-1}r_0^{N-1}}{r_0^{N(m-1)}}\lesssim 1$.

Depending on our choice of connecting model, that is catenoid or floating disk, we choose $u_1=c_\rho$ or $u_1=\tilde{c}_\rho$ respectively. Now we may apply the construction of the previous section, Proposition \ref{prop:localglue} and Lemma \ref{lem.definition of the metric}, to obtain that properties \ref{list.properties of v} and \ref{list.properties of omega} are satisfied.

It remains to prove \ref{list.properties of lambda} by connecting the conformal factor $\lambda$ of Lemma \ref{lem.definition of the metric} in $y_{m+1}>0$ to $\lambda\circ \boldsymbol{s}$ in $\{y_{m+1}<0\}$ inside $\bC_{12}\setminus \bC_7$, where we glued $u_1$ to $u_2$. To do this, we define the Lipschitz continuous 
\[ \lambda_1(y) = \begin{cases}
    \lambda(y) &\text{ if } y_{m+1}>0\\
    \lambda \circ \boldsymbol{s}(y) &\text{ if } y_{m+1}<0
\end{cases}\,.\]
Note that $\lambda_1$ is smooth in all the directions except $e_{m+1}$, i.e. $\partial^\alpha\lambda_1$ exists for all $\alpha$ with $\alpha_{m+1}=0$ and its Lipschitz constant is bounded by $H^{m-1}$ due to \ref{list.properties of lambda}. 

Let us fix a symmetric bump function $ \epsilon_0(t)$ supported in $(-1,1)$ with $\epsilon_0 (t)=1$ for $|t|\le \frac12$. Now we set $\epsilon(y_{m+1})= H \epsilon_0( \frac{8y_{m+1}}{cH})$ where $c>0$ is the constant in \eqref{list.properties of v} with $v\ge c H$. Hence we have $|\epsilon^{(k)}(y_{m+1})| \lesssim_k H^{1-k}$ with $\epsilon(y_{m+1})=H$ if $|y_{m+1}|\le \frac{c}{2}H$.  For a one-dimensional convolution kernel $\varphi(t)$ with $\spt(\varphi)\subset (-1,1)$ we define the conformal factor in $\bC_{12}\setminus \bC_7$ to  be 
\begin{equation}\label{eq.conformal factor}
    \lambda_2(x,y_{m+1})= \int \varphi(s)\, \lambda_1(x,y_{m+1} + \epsilon(y_{m+1})s)\,ds\,.
\end{equation}
Now we can compute and estimate its derivatives.
We distinguish three regions $|y_{m+1}|>cH$, $\frac{c}{4} H \le |y_{m+1}| \le cH$ and $|y_{m+1}| \le \frac{c}{4}H$.
In the first region, we restrict ourself to $y_{m+1}$, $\epsilon(y_{m+1})=0$ and so $\lambda_2=\lambda$ hence $|D^k\lambda(y)|\lesssim_k H^{m-1}$.

In the second region, $\frac{c}{4} H \le y_{m+1} \le c H$, by our choice of $\epsilon$ we have that $\lambda_1=\lambda$ in \eqref{eq.conformal factor} and hence for any $\alpha\in \N_0^m, k\in \N_0$ 
\begin{align*}
    \partial_x^\alpha \partial_{y_{m+1}}^k\lambda_2(y) = \int \varphi(s)\, \partial_{y_m+1}^k \left( \partial_x^\alpha\lambda(x,y_{m+1} + \epsilon(y_{m+1})s)\right)\, ds 
\end{align*}
Using the  Fa\`a di Bruno, the ``worst'' term to estimate, given our bounds, is
\[
|\partial_x^\alpha\partial_{y_{m+1}}\lambda(x,y_{m+1} + \epsilon(y_{m+1})s)\epsilon^{(k)}(y_{m+1})|\lesssim_{k+|\alpha|} H^{m-1} H^{1-k}\, ,
\]
and hence we conclude  the argument in the second region.


Finally in the third region, $|y_{m+1}|\le \frac{c}{4} H$, where $\epsilon(y_{m+1})=H$ is constant, hence we can use the classical convolution estimates to deduce that 
\[
\bigl|\partial_x^\alpha \partial_{y_{m+1}}^k\lambda_2(y)\bigr| =  \Big|\int \partial^{k-1}_{y_{m+1}}\varphi_H(y_{m+1}-s)\, \partial_x^\alpha\partial_{y_{m+1}}\lambda(x,s)\, ds  \Big|\lesssim_{k+|\alpha|} H^{1-k} H^{m-1} \,.
\]

\end{proof}

\subsection{The iterative gluing}\label{subsec.iterative gluing}

To construct our example, let us fix a power $1<p<N$, e.g. $p=2$, and fix the choice $r=r_0^p$, i.e. $H\approx r_0^{N-p}$. Furthermore, fix a sequence $r_{0,k} \downarrow 0$ with associated heights $H_k\approx r_{0,k}^{N-p}$. The chosen sequence should satisfy $r_{0,{k+1}}-r_{0,k} \gg r_{0,k}^p$ and $\sum_{k} r_{0,k}^{m-1} < \infty$, e.g. $r_{0,k}=2^{-k}$. 
We consider the sequence of cylinders $B_k\times \R=\bC_k=\bC_{r_{0,k}^p}(r_{0,k} e_1)$ and for each of them let $\underline{B}_k \times \R=\underline{\bC}_k=\bC_{\theta_k H_k}(r_{0,k} e_1)$ be the smaller cylinder where $\theta_k H_k$ is the radius of the hole/ disk of the chosen model with height $H_k$. By an abuse of notation we will denote with $c B$ the ball  with the same center but with radius $c\operatorname{radius}(B)$ and the same in case of cylinders. 

Now in each, of the cylinders $12 \bC_k$ we perform the single step gluing.
In fact, we construct a sequence of functions $u_k \colon U_k \to \R$ on perforated domains $U_k\subset \R^m$, a sequence of associated varifolds $V_k$, a sequence of forms $\omega_k$ and a sequence of conformal factors $\lambda_k$. In the case of the catenoid as gluing model, $V_k$ is the varifold associated to a current $T_k$ with $\partial T_k =0$. 

To be more precise, let $u_0 = u$ be the function of the base model, $U_0=\R^{m}$, $\omega_0=\omega_{u_0}$ and $\lambda_0 =1$. Furthermore, we set $T_0=\a{\bG_{u_0}} + \boldsymbol{s}_\sharp \a{\bG_{u_0}}$  in the case of Theorem \ref{t:uno} and $T_0=\a{\bG_{u_0}} - \boldsymbol{s}_\sharp \a{\bG_{u_0}}$ in the case of Theorem \ref{t:due}, with associated varifold $V_0$.

Now assume that the elements in the $k-1$ step have been constructed. 
We set $\eta_k=\eta_{r_{0,k}e_1,r_{0,k}^p}$ and apply Proposition \ref{prop:single gluing} with respect to the parameters $r_0=r_{0,k}, r=r_{0,k}^p$ to obtain function, form and conformal factor, $v,\omega, \lambda$.  The $k$th objects are then obtained as follows
\begin{itemize}
    \item The perforated domain is $ U_k=U_{k-1}\setminus \underline{B}_k$.
    \item The function $u_k\colon U_k \to \R$ is given by 
    \[u_k=\begin{cases}
        u_{k-1} \text{ on } U_k\setminus 12B_k\\
        r v\circ \eta_k \text{ on } U_k\cap 12B_k 
    \end{cases}\,.\]
Note that it satisfies $u_k|_{12B_k^c}=u_l|_{12B_k^c}$ for every $l>k$ and moreover
    \[\norm{u_k-u_{k-1}}_{C^0(\overline{U}_k)} \lesssim 2r H_k\lesssim r_{0,k}^N\,.\]
    \item  The currents are
    \[ T_k=\begin{cases} \a{\bG_{u_k}\res (U_k\times \R)} + \boldsymbol{s}_\sharp\a{\bG_{u_k}\res U_k\times \R} &\text{ in the case of the catenoid model }\\
    \a{\bG_{u_k}\res (U_k\times \R)}  - \boldsymbol{s}_\sharp\a{\bG_{u_k}\res U_k\times \R} + \sum_{l \le k} \a{\underline{B}_k\times \{0\}} &\text{ in the case of floating disks,} \end{cases} \]
    having as boundaries $\partial T_k=0$ in the case of catenoids and $\partial T_k = \sum_{l \le k } \a{ \partial \underline{B}_k \times \{0\}}$ in the case of ``floating disks''. Hence we have the mass estimate
    \[\norm{T_k}(\bC_R)\lesssim \norm{T_{k-1}}(\bC_R)+ C r_{0,k}^m\,, \]
    and in case of the ``floating'' disk the mass estimate of the boundary
    \[\norm{\partial T_k}\lesssim \norm{\partial T_{k-1}} + |\partial B_1|\, (\theta_kH_k)^{m-1} \lesssim c\sum_{l\le k}r_{0,k}^{m-1}\,.\]
    
    \item The calibration form, using the fact that $\dist(\eta_k(y),\bG_v) = \frac{1}{r} \dist(y, {\eta_k^{-1}}_\sharp \bG_v)$ and for $y\in \bC_k$ we have $r\approx |x|^p$, is defined by
    \[\omega_k = \begin{cases}
        \omega_{k-1} &\text{ in } (U_k\setminus 12B_k)\cap \{ \dist(y,{\rm spt}(\bG_{u_k}))\le \delta |x|^p\}\\
        \eta_k^\sharp \omega &\text{ in }\eta_k^{-1}(B_{12}\setminus B_{\theta_kH_k}) \cap \{ \dist(y,{\rm spt}(\bG_{u_k}))\le \delta |x|^p\}\,.
    \end{cases}\]
    \item The conformal factor is defined by
    \[
    \lambda_k=
    \begin{cases}
    \lambda_{k-1} & \text{ in } 12\bC_k^c \\
    \lambda \circ \eta_k &\text{ in }12\bC_k\,, 
    \end{cases}
    \]
    and it satisfies 
    \[\norm{\lambda_k-\lambda_{k-1}}_{C^j} = \norm{1-\lambda\circ\eta_k}_{C^j(12\bC_k\setminus 7\bC_k)}\le r^{-j} \norm{1-\lambda\circ\eta_k}_{C^j(\bC_{12}\setminus \bC_{7})}\lesssim_j r_{0,k}^j H_k^{m-j}\,.\]
\end{itemize}

 It follows that the conformal factors $\lambda_k$ converges to $\lambda$ in $C^j$ for all $j<m$ such that $r_{0,k}^{-pj} H_k^{m-j} \downarrow 0$, which is satisfied for $m(N-p)\ge jN$. 
In fact by interpolation we have 
\[ \norm{\lambda_k - \lambda}_{C^{j,\alpha}} \lesssim \norm{\lambda_k-\lambda}_{C^j}^{1-\alpha} \norm{\lambda_k-\lambda}_{C^{j+1}}^{\alpha}\lesssim r_{0,k}^{(1-\alpha)(m(N-p)-jN)+\alpha(m(N-p)-(j+1)N)}\downarrow 0\]
if $m(N-p) > (j+\alpha)N$. Hence for any $\alpha<1$ we can find a sufficiently large $N$ such that $N(m-1 +\alpha)< m(N-p)$. When $m=1$, we need $N\alpha<N-p$ and also, from Proposition \ref{prop:single gluing}, $p>1$, which can both be satisfied for sufficiently large $N$ depending on $\alpha\in (0,1)$. However, for $m\geq 2$, it is easy to see that no such choice of $N$ works. Moreover $\lambda=\lambda \circ \boldsymbol{s}$ and we set $g=\lambda\,\delta_{\R^{m+1}}$.

Next we consider the perforated domain 
\[U:=\R^m\setminus \bigcup_k\underline{B}_k=\lim_k U_k\]
The functions $u_k\colon U \to \R$ converge uniformly in $C^1_{loc}(\overline{U}\setminus \{0\})\cap C^0_{loc}(\overline{U})$ to a $C^1$ function $u$, since they stabilize. Moreover, taking into account the previously established mass bounds, by the compactness of integral currents,  the currents $T_k$ converge to a current  $T$ which is given by 
\[ T=\begin{cases}\a{\bG_{u}\res (U\times \R)} + \boldsymbol{s}_\sharp\a{\bG_{u}\res U\times \R} &\text{ in the case of the catenoid model }\\
    \a{\bG_{u}\res (U\times \R)}  - \boldsymbol{s}_\sharp\a{\bG_{u}\res U\times \R} + \sum_{k} \a{\underline{B}_k\times \{0\}} &\text{ in the case of floating disks}\end{cases}\]

The forms $\omega_k$ converge locally uniformly in $C^1$ on $(U\times \R)\cap \{\dist(y,\bG_{u})<\delta\,|x|^p\}$ to a close form $\omega$, such that $\operatorname{comass}_g(\omega)\leq 1$ where it is defined, and moreover $\omega=d\,\operatorname{vol}_u$ on $\bG_u$.
In the case of the floating disk we can also consider the associated varifold 
\[
V= \a{\bG_{u}\res (U\times \R)} + \boldsymbol{s}_\sharp\a{\bG_{u}\res U\times \R} +  \sum_k \a{\underline{B}_k\times \{0\}}\,.
\]

Finally we observe that, since $\lambda=\lambda \circ \boldsymbol{s}$, then both $\a{\bG_u\res(U\times \R)}$ and $\boldsymbol{s}_\sharp\a{\bG_u\res(U\times \R)}$ are calibrated by $\omega$ and $\boldsymbol{s}^\sharp\omega$ in the metric $g=\lambda^2 \delta_{\R^{m+1}}$.

\subsection{Stationarity}\label{subsec.minimality}
Since stationarity is a local property, we can use a partition of unity of to check it: Consider a open covering of $\R^{m+1}$ by $\{B_k\times (-\frac1k,\frac1k)\}_k$ and $V= \R^m\setminus \bigcup_k \left(\frac34\overline{B_k}\times[-\frac{1}{2k}, \frac{1}{2k}]\right)$ with subordinate partition of unity $\{\theta_k\}_k \cup \{\theta_0\}$. 

Given a smooth vector field $X$ compactly supported in $\R^{m+1}\setminus \{0\}$ consider its finite sum decomposition 
\[ X = \theta_0X + \sum_{k<N} \theta_k X = X_0 + \sum_{k <N} X_k\,. \]
The minimality follows if we show that the first variation vanishes for each of these vector fields separately. 

Firstly, since $\spt(X_0) \cap \bigcup_k \left(\frac34\overline{B_k}\times[-\frac{1}{2k}, \frac{1}{2k}]\right) = \emptyset$ there exists a $t_0>0$ such that for all $|t|<t_0$ $\spt( {\phi_t}_\sharp \bG_u) \subset \{y \colon \dist(y,\bG_{u})<\delta\,|x|^p\}$ and $\spt( {\phi_t}_\sharp \boldsymbol{s}_\sharp \bG_u) \subset \{y \colon \dist(y,\boldsymbol{s}_\sharp\bG_{u})<\delta\,|x|^p\}$. Hence using that  $\spt(\partial \bG_u) \cap \spt(X_0) =\emptyset$ and $\bG_u$, $\boldsymbol{s}_\sharp \bG_u$ are calibrated by $\omega$ and $\boldsymbol{s}^\sharp\omega$ respectively we deduce
\[ \norm{{\phi_t}_\sharp \bG_u}-\norm{\bG_u} \ge ({\phi_t}_\sharp\bG_u - \bG_u)(\omega)\ge 0 \]
and the same for $\operatorname{s}_\sharp \bG_u$. This shows the  stationarity on $\spt(X_0)$. 

Secondly, since for any $k$ we have $T\res \bC_k = {\eta_k^{-1}}\rho \operatorname{Cat}$ in case of the catenoid, we deduce that $T$ is  stationary with respect to variations of $X_k$.

Thirdly, since for any $k$ we have $V\res \bC_k = {\eta_k^{-1}}\Sigma_{\overline{\rho}_k}$ in case of the floating disk we deduce that $V$ is stationary with respect to variations of $X_k$. 

Thus in summary, we conclude the stationarity of the varifold in both theorems.

\subsection{Stability in case of floating disks}\label{subsec.stability} 
Let $K$ be any compact set in $\R^{m+1}\setminus \left(\{0\} \cup \bigcup_k \partial \underline{B}_k \times \{0\} \right)$. Then 
\begin{equation}\label{e:formula-varifold}
 V \res K = \left( (\bG_u +\boldsymbol{s}_\sharp \bG_u )\res K) \right) +  \sum_{k \le N} \left(\a{\underline{B}_k\times \{0\} }\res K\right)\,.
\end{equation}
Here with a slight abuse of notation we use $\bG_u$ for the {\em varifold} induced by the graph\footnote{ Following \cite{Simon} we consider an integer rectifiable varifold as a nonnegative Radon measure induced by the Hausdorff measure restricted on a rectifiable set and weighted with an integer-valued Radon function. This explains the plus sign in front of $\boldsymbol{s}_\sharp \bG_u$ in the formula \eqref{e:formula-varifold}, rather than the minus sign in the previous formulae, when we were dealing with the currents.}.
The support of each of the summands in \eqref{e:formula-varifold} has a positive distance from the union of the supports of the other summands. Furthermore, the first two summands is varifolds induced by graphs, and the others are induced by flat embedded disks. 
Additionally, regarding each summand as a current, their boundaries all vanish on $K$ 
\[ (\partial \bG_u)\res K =-(\partial \boldsymbol{s}_\sharp\bG_u)\res K =0 \text{ and } (\partial \a{\underline{B}_k\times \{0\}}) \res K=0\,.\]
Arguing as above for the minimality we deduce that the summands $\bG_u$, $\boldsymbol{s}_\sharp\bG_u$ are stable varifolds in $K$, and since each disk is minimizing with respect to its boundary, they are stable in $K$ as well. Hence $V\res K$ is a stable immersion in the sense of \cite{Bellettini2025}.





\appendix
\section{A variant of the Poincare lemma}
In this section we present a variant of the classical Poincar\'e lemma  for differential forms. This formula allows us to get better estimates for the metrics in our constructions in a portion of the domain, compared to the usual one. Since the final regularity of the metrics $g$ in both theorems is driven by the worse estimates in other regions, we could in fact dispense from this variant and use the more classical one. We however believe that the idea has an independent interest. It exploits Cartan's homotopy formula and the relation between the Lie-derivative and the exterior differential, i.e.  
\begin{equation}\label{eq.Lie derivative}
    \mathcal{L}_Y\omega = d\left(i_Y\omega\right)+i_Y(d\omega)
\end{equation}
where $i_Y\colon\Lambda^k(M) \to \Lambda^{k-1}(M)$ is the contraction of the differential form $\omega$ with the vector field $Y$.
\begin{lemma}\label{lem.poincare}
Given a smooth vector field $Y$, assume that the associated flow $\phi_t$ exists for all $t \in (a,0]$. Then the linear operator $I^Y\colon \Lambda^k(M) \to \Lambda^{k-1}(M)$ given by 
\begin{equation}\label{eq.poincare formula}
I^Y\omega = \int_a^0 \phi_t^\sharp (i_Y\omega)\, dt
\end{equation}
satisfies 
\begin{equation}\label{eq.poincare relation}
    dI^Y\omega + I^Yd\omega= \omega- \phi^\sharp_a\omega\,.
\end{equation}
\end{lemma}
\begin{proof}
This is a direct computation: 
\begin{align*}
     dI^Y\omega + I^Yd\omega = \int_a^0 \phi_t^\sharp\left( d\left(i_Y\omega\right)+i_Y(d\omega)\right)\,dt= \int_a^0 \phi_t^\sharp\left( \mathcal{L}_Y\omega\right)\,dt=\int_a^0 \frac{d}{dt} \phi_t^\sharp\omega\,dt= \omega- \phi^\sharp_a\omega\,.
\end{align*}
\end{proof}
Note that the classical Poincar\'e lemma on a starhaped domain $U\subset \R^n$ is in fact the particular case of the above formula given when the vector field $Y(x)$ is $x$ (assuming that $U$ is star-shaped with respect to the origin). The flow is then given by $\phi_t(x)=e^{t}x$,  i.e. it exists for all $t<0$, and we may choose $a=-
\infty$ in \eqref{eq.poincare formula} to conclude the classical statement.

\bibliographystyle{plain}
\bibliography{lit}

@article{Bellettini2025,
 author = {Bellettini, Costante},
 title = {Extensions of {Schoen}-{Simon}-{Yau} and {Schoen}-{Simon} theorems via iteration {\`a} la {De} {Giorgi}},
 fjournal = {Inventiones Mathematicae},
 journal = {Invent. Math.},
 issn = {0020-9910},
 volume = {240},
 number = {1},
 pages = {1--34},
 year = {2025},
 language = {English},
 doi = {10.1007/s00222-025-01317-0},
 zbMATH = {7991638}
}

@article {All,
    AUTHOR = {Allard, William K.},
     TITLE = {On the first variation of a varifold},
   JOURNAL = {Ann. of Math. (2)},
  FJOURNAL = {Annals of Mathematics. Second Series},
    VOLUME = {95},
      YEAR = {1972},
     PAGES = {417--491},
      ISSN = {0003-486X},
   MRCLASS = {49F20},
  MRNUMBER = {307015},
MRREVIEWER = {M.\ Klingmann},
       DOI = {10.2307/1970868},
       URL = {https://doi.org/10.2307/1970868},
}

@article{CaffarelliHardtSimon1984,
 author = {Caffarelli, Luis and Hardt, Robert and Simon, Leon},
 title = {Minimal surfaces with isolated singularities},
 fjournal = {Manuscripta Mathematica},
 journal = {Manuscr. Math.},
 issn = {0025-2611},
 volume = {48},
 pages = {1--18},
 year = {1984},
 language = {English},
 doi = {10.1007/BF01168999},
 url = {https://eudml.org/doc/155017},
 zbMATH = {3907311},
 Zbl = {0568.53033}
}

@misc{BMW,
      title={On the Nature of Stationary Integral Varifolds near Multiplicity 2 Planes}, 
      author={Spencer Becker-Kahn and Paul Minter and Neshan Wickramasekera},
      year={2025},
      eprint={2507.13148},
      archivePrefix={arXiv},
      primaryClass={math.DG},
      url={https://arxiv.org/abs/2507.13148}, 
}

@unpublished{CCSvarifolds,
	author = {Camillo Brena and Stefano Decio and Camillo De Lellis},
	title = {{Remarks and Conjectures on Stationary Varifolds}},
    note = {Preprint, arXiv: 2503.00651},
	year = {2025}}

@article {SW,
    AUTHOR = {Simon, Leon and Wickramasekera, Neshan},
     TITLE = {A frequency function and singular set bounds for branched
              minimal immersions},
   JOURNAL = {Comm. Pure Appl. Math.},
  FJOURNAL = {Communications on Pure and Applied Mathematics},
    VOLUME = {69},
      YEAR = {2016},
    NUMBER = {7},
     PAGES = {1213--1258},
      ISSN = {0010-3640,1097-0312},
   MRCLASS = {58E12 (35B65 53C42)},
  MRNUMBER = {3503021},
MRREVIEWER = {Futoshi\ Takahashi},
       DOI = {10.1002/cpa.21642},
       URL = {https://doi.org/10.1002/cpa.21642},
}

@article {MiWi,
    AUTHOR = {Minter, Paul and Wickramasekera, Neshan},
     TITLE = {A structure theory for stable codimension 1 integral varifolds
              with applications to area minimising hypersurfaces {${\rm
              mod}\,p$}},
   JOURNAL = {J. Amer. Math. Soc.},
  FJOURNAL = {Journal of the American Mathematical Society},
    VOLUME = {37},
      YEAR = {2024},
    NUMBER = {3},
     PAGES = {861--927},
      ISSN = {0894-0347,1088-6834},
   MRCLASS = {53A10},
  MRNUMBER = {4736529},
MRREVIEWER = {Fei-Tsen\ Liang},
       DOI = {10.1090/jams/1032},
       URL = {https://doi.org/10.1090/jams/1032},
}

@misc{BDF,
      title={$C^\infty$ rectifiability of stationary varifolds}, 
      author={Camillo Brena and Camillo De Lellis and Federico Franceschini},
      year={2025},
      eprint={2503.00649},
      archivePrefix={arXiv},
      primaryClass={math.AP},
      url={https://arxiv.org/abs/2503.00649}, 
}

@incollection {SavinAllard,
    AUTHOR = {Savin, O.},
     TITLE = {Viscosity solutions and the minimal surface system},
 BOOKTITLE = {Nonlinear analysis in geometry and applied mathematics. {P}art
              2},
    SERIES = {Harv. Univ. Cent. Math. Sci. Appl. Ser. Math.},
    VOLUME = {2},
     PAGES = {135--145},
 PUBLISHER = {Int. Press, Somerville, MA},
      YEAR = {2018},
      ISBN = {978-1-57146-359-3},
   MRCLASS = {35J93 (35B65 35D40 53C42)},
  MRNUMBER = {3823884},
MRREVIEWER = {Naoki\ Yamada},
}

@article {MenneJGA,
    AUTHOR = {Menne, Ulrich},
     TITLE = {Second order rectifiability of integral varifolds of locally
              bounded first variation},
   JOURNAL = {J. Geom. Anal.},
  FJOURNAL = {Journal of Geometric Analysis},
    VOLUME = {23},
      YEAR = {2013},
    NUMBER = {2},
     PAGES = {709--763},
      ISSN = {1050-6926,1559-002X},
   MRCLASS = {49Q15 (35J60)},
  MRNUMBER = {3023856},
       DOI = {10.1007/s12220-011-9261-5},
       URL = {https://doi.org/10.1007/s12220-011-9261-5},
}

@article{DelellisAllard,
	author = {De Lellis, Camillo},
	journal = {Proceedings of CMSA Harvard. Nonlinear analysis in geometry and applied mathematics. Part 2},
	pages = {23--49},
	title = {Allard's interior regularity theorem: an invitation to stationary varifolds},
	url = {http://cvgmt.sns.it/paper/3454/},
	year = {2018}}

@article {DPGS,
    AUTHOR = {De Philippis, Guido and Gasparetto, Carlo and Schulze, Felix},
     TITLE = {A short proof of {A}llard's and {B}rakke's regularity
              theorems},
   JOURNAL = {Int. Math. Res. Not. IMRN},
  FJOURNAL = {International Mathematics Research Notices. IMRN},
      YEAR = {2024},
    NUMBER = {9},
     PAGES = {7594--7613},
      ISSN = {1073-7928,1687-0247},
   MRCLASS = {49Q20 (49N60 49Q05 49Q15 53E10)},
  MRNUMBER = {4742836},
       DOI = {10.1093/imrn/rnad281},
       URL = {https://doi.org/10.1093/imrn/rnad281},
}

@misc{HS,
      title={Measure $0$ of the singular set for $2$-valued stationary hypercurrents}, 
      author={Jonas Hirsch and Luca Spolaor},
      year={2025},
      eprint={2509.05544},
      archivePrefix={arXiv},
      primaryClass={math.DG},
      url={https://arxiv.org/abs/2509.05544}, 
}

@book{Brakke,
    AUTHOR = {Brakke, Kenneth A.},
     TITLE = {The motion of a surface by its mean curvature},
    SERIES = {Mathematical Notes},
    VOLUME = {20},
 PUBLISHER = {Princeton University Press, Princeton, NJ},
      YEAR = {1978},
     PAGES = {i+252},
      ISBN = {0-691-08204-9},
   MRCLASS = {49F22 (35K99 49F20 58D25)},
  MRNUMBER = {485012},
MRREVIEWER = {Jean\ E.\ Taylor},
}

@misc{HirschSpolaor2024,
      title={Dimension of the singular set for $2$-valued stationary Lipschitz graphs}, 
      author={Jonas Hirsch and Luca Spolaor},
      year={2023},
      eprint={2401.00279},
      archivePrefix={arXiv},
      primaryClass={math.DG},
      url={https://arxiv.org/abs/2401.00279}, 
}

@article{DHMSS,
  title={Area minimizing hypersurfaces modulo $ p $: a geometric free-boundary problem},
  author={De Lellis, Camillo and Hirsch, Jonas and Marchese, Andrea and Spolaor, Luca and Stuvard, Salvatore},
  journal={accepted in Journal of Functional Analysis},
  year={2025}
}

@article{DDH,
    author = {De Lellis, Camillo and De Philippis, Guido and Hirsch, Jonas},
    title = {Nonclassical minimizing surfaces with smooth boundary},
    journal = {J. Differential Geom.} ,
    year = {2022},
    volume ={122},
    edition ={2},
    pages ={205 -- 222}
}

@book{Simon,
	Author = {Simon, Leon},
	Isbn = {0-86784-429-9},
	Mrclass = {49-01 (28A75 49F20)},
	Mrnumber = {756417 (87a:49001)},
	Mrreviewer = {J. S. Joel},
	Pages = {vii+272},
	Publisher = {Australian National University, Centre for Mathematical Analysis, Canberra},
	Series = {{Proceedings of the Centre for Mathematical Analysis, Australian National University}},
	Title = {{Lectures on geometric measure theory}},
	Volume = {3},
	Year = {1983}}
\end{document}